\documentclass[a4paper,draft,reqno,12pt]{amsart}
\usepackage[english]{babel}

\usepackage[utf8]{inputenc} 
\usepackage[utf8]{inputenc}
\usepackage[T2A]{fontenc}
\usepackage{amsfonts,amsmath,amssymb,amsthm}

\usepackage{microtype}
\usepackage{mathtools}
\usepackage{geometry}
\usepackage{textcomp}
\usepackage[usenames]{color}
\usepackage{faktor}
\usepackage{xfrac}
\usepackage{euscript}
\usepackage{tikz}
\usepackage{pst-node}
\usepackage{tikz-cd} 
\usepackage{enumitem}
\makeatletter
\AddEnumerateCounter{\asbuk}{\russian@alph}{щ}
\makeatother

\newcommand*{\hm}[1]{#1\nobreak\discretionary{}%
{\hbox{$\mathsurround=0pt #1$}}{}}

\theoremstyle{theorem}
\newtheorem{theorem}{Theorem}
\newtheorem{lemma}{Lemma}
\newtheorem{defn}{Definition}

\newtheorem{cor}{Corollary}

\newtheorem{stat}{Proposition}
\theoremstyle{definition}

\newtheorem{constr}{Construction}

\theoremstyle{theorem}
\newtheorem{example}{Example}

\theoremstyle{theorem}
\newtheorem{remark}{Remark}

\DeclareMathOperator{\Spec}{Spec}
\DeclareMathOperator{\Hom}{Hom}

\DeclareMathOperator{\SL}{SL}

\DeclareMathOperator{\Aut}{Aut}
\DeclareMathOperator{\Ker}{ker}
\DeclareMathOperator{\im}{im}

\DeclareMathOperator*{\Prod}{\prod}
\DeclareMathOperator{\Quot}{Quot}

\DeclareMathOperator{\Lin}{span}

\DeclareMathOperator{\degg}{\overline{deg}}
\DeclareMathOperator{\deggg}{\widehat{deg}}

\title{HOMOGENEOUS LOCALLY NILPOTENT DERIVATIONS ON TRINOMIAL VARIETIES}
\author{Kirill Rassolov}
\thanks{This article is an output of a research project implemented as part of the Basic Research Program at HSE University.
\\ This is a preprint of the Work accepted for publication in Mathematical Notes, 2025, vol. 118, no.~3-4; https://pleiades.online/}
\email{kirill.rassolov@math.msu.ru}
\address{
HSE University, Faculty of Computer Science, Pokrovsky Boulevard 11, Moscow 109028, \mbox{Russia};
\linebreak
and
\linebreak
Lomonosov Moscow State University, Faculty of Mechanics and Mathematics, Department of Higher Algebra, Leninskie Gory 1, Moscow 119234, Russia
}

\begin{document}
\maketitle

\begin{abstract}
    We consider the finest grading of the algebra of regular functions of a trinomial variety. An explicit description of locally nilpotent derivations that are homogeneous with respect to this grading is obtained.
\end{abstract}

\section{Introduction}
A trinomial variety is an affine algebraic variety given by a system of equations of the form
\begin{equation}
a + bT_{11}^{l_{11}}\cdots T_{1n_1}^{l_{1n_1}} + cT_{21}^{l_{21}}\cdots T_{2n_2}^{l_{2n_2}}=0
\end{equation}
or
\begin{equation}
    aT_{01}^{l_{01}}\cdots T_{0n_0}^{l_{0n_0}} + bT_{11}^{l_{11}}\cdots T_{1n_1}^{l_{1n_1}} + cT_{21}^{l_{21}}\cdots T_{2n_2}^{l_{2n_2}}=0.
\end{equation}
An example of a trinomial variety is provided by the special linear group $\SL_2(\mathbb K)$, given by~$T_{11}T_{12} \hm- T_{21}T_{22} = 1$; see Construction~1 for precise definition.

Trinomial varieties arise from the Cox construction, which defines the total coordinate ring $\mathcal R (X)$ for a normal variety $X$ with a finitely generated divisor class group (see~\cite{Ar} or~\cite{BH}). Recall that the complexity $c(X)$ of an action of a reductive group $G$ on $X$ is defined as the codimension of a generic orbit of the induced action of a Borel subgroup $B\subseteq G$. Given an effective action of an algebraic torus $T$, we have $c(X) = \dim X - \dim T$. In particular, toric varieties correspond to $c(X)=0$. It is well known that the Cox ring of a toric variety is a polynomial ring. It is proved in~\cite{HW} and~\cite{HS} that in the case of an effective torus action of complexity one on $X$, the total coordinate space $\Spec \mathcal R(X)$ is a trinomial variety. As total coordinate spaces trinomial varieties are also investigated in \cite{HH} and \cite{AHHL}.

We work over an algebraically closed field $\mathbb{K}$ of characteristic zero. Let $R$ be a commutative $\mathbb{K}$-domain.
By definition, a \textit{derivation} is a $\mathbb{K}$-linear map $\delta\colon R\to R$ satisfying the Leibniz rule: $\delta(ab) = \delta(a)b + a\delta(b)$ for all $a,b\in R$. A derivation $\delta$ is said to be \textit{locally nilpotent} if for any $a\in R$ there exists $n\in \mathbb Z_{\ge 0}$ such that $\delta^n(a)=0$. Let $R$ be graded by a group $G$; that is to say,
$$
R = \bigoplus_{g\in G} R_g
$$
for some $\mathbb K$-subspaces $R_g\subseteq R$ with $R_g R_h \subseteq R_{g+h}$ for any $g,h\in G$. A derivation $\delta\colon R\to R$ is called {\it homogeneous} if there is $d\in G$ with $\delta (R_g) \subseteq R_{g+d}$ for all $g\in G$. In this case, $d$ is called the \emph{degree} of~$\delta$.

We investigate the case where $R=\mathbb K[X]$ is the algebra of regular functions of a trinomial variety $X$. Since $R$ is a finitely generated commutative algebra, we can suppose that G is a finitely generated abelian group. We consider the finest grading on $R$ such that all coordinate functions are homogeneous. This grading was introduced in~\cite[Construction~1]{HW}. It corresponds to an effective action on $X$ of some quasitorus $H$. Our main result is an explicit description of locally nilpotent derivations that are homogeneous with respect to the finest grading. In geometric terms, any such derivation matches an $H$-normalized action on X of the additive group $\mathbb G_a = (\mathbb K, +)$. In~\cite{GZ} and~\cite{Z}, the desired description was obtained for trinomial hypersurfaces given by an equation of the form~(2). 
The present paper generalizes that description to an arbitrary trinomial variety. 

Our description consists of two parts. First, we prove that all suitable irreducible derivations are elementary (the definition is given in Section~3). For the second part, we describe the kernel of an irreducible derivation in Section~6. The resulting description is summarized in Section~6, Corollary~\ref{Cor2}.

As an application, we notice the following interesting phenomenon. Recall that a variety~$X$ that admits no nontrivial $\mathbb G_a$-actions is said to be \textit{rigid}. The action of the quasitorus $H$ induces the action of the torus $T=H^0\curvearrowright X$. The last action corresponds to the grading of $\mathbb K[X]$ by some free abelian group. In this case, it is well known that the existence of $\mathbb G_a$-actions is equivalent to the existence of $T$-normalized $\mathbb G_a$-actions. However, this is not true for the quasitorus $H$. Namely, Corollary~\ref{Cor1} gives a criterion for a non-rigid trinomial variety to have no $H$-normalized $\mathbb G_a$-action.

The author is grateful to Sergey Gaifullin for statement the problem and useful discussions.

\section{Trinomial varieties}
Following~\cite[Construction 1]{HW}, we give a definition of a trinomial variety and the finest grading in constructive way.

\begin{constr}\label{ConstrR}
Set integers $r,n>0$, $m\ge0$ and a partition $n=n_{\iota}+\ldots+n_r$, $n_i>0$. For every $i=\iota,\ldots, n$ fix a tuple $l_i=(l_{i1},\ldots,l_{in_i})\in\mathbb{N}^{n_i}$ and denote the monomial $$T^{l_i}_i = T^{l_{i1}}_{i1}\cdots T^{l_{in_i}}_{in_i}\quad\in\quad\mathbb{K}[T_{ij}, S_k \;| \;\iota\le i\le n,\; 1\le j\le n_i,\; 1\le k\le m].$$ 
The above-mentioned polynomial algebra is further denoted by $\mathbb{K}[T_{ij}, S_k].$
We consider two types of trinomial varieties given by equations of different form.

\textit{Type 1.} Set $\iota=1,\; I=\{1,\ldots,r-1\}$. Fix a tuple $A = (a_1,\ldots,a_r)\in\mathbb{K}^r$ of pairwise different elements of $\mathbb K$. For every $i\in I$ define a polynomial $$\mathfrak g_i=T^{l_i}_i - T^{l_{i+1}}_{i+1} - (a_{i+1}-a_i).$$

Build an $r\times(n+m)$-matrix  consisting of rows $l_1,\ldots,l_r$:
$$
P_0 = \begin{pmatrix}
l_1 & \ldots & 0 & 0 & \ldots & 0\\\
\vdots & \ddots & \vdots & \vdots & & \vdots\\
0 & \ldots & l_r & 0 & \ldots & 0\\
\end{pmatrix}.
$$

\textit{Type 2.} Set $\iota=0,\; I=\{0,\ldots,r-2\}$. Fix a matrix $A=(a_0,\ldots,a_r)$ with pairwise linearly independent columns $a_0,\ldots,a_r\in\mathbb{K}^2$. For every $i\in I$ define a polynomial $$\mathfrak g_i = \det{\begin{pmatrix}
T^{l_i}_i & T^{l_{i+1}}_{i+1} & T^{l_{i+2}}_{i+2}\\
a_i & a_{i+1} & a_{i+2}\end{pmatrix}}.$$

Build an $r\times(n+m)$-matrix, consisting of strings $l_0,\ldots,l_r$:
$$
P_0 = \begin{pmatrix}
-l_0 & l_1 & \ldots & 0 & 0 & \ldots & 0\\
\vdots & \vdots & \ddots & \vdots & \vdots & & \vdots\\
-l_0 & 0 & \ldots & l_r & 0 & \ldots & 0\\
\end{pmatrix}.
$$

A \textit{trinomial variety} $X$ is a subvariety of $\mathbb{A}^{n+m}$ given by a system of equations $\mathfrak g_i = 0, \,i\in I$. It follows from Proposition~\ref{PropHausen} below that the algebra
$$
R(A, P_0) := \mathbb{K}[T_{ij}, S_k]/ \mathfrak{I},\quad\text{where}\quad \mathfrak{I}=\mathfrak I(A, P_0) = (\mathfrak g_i\; |\;i\in I),
$$
has no nilpotents and therefore $R(A, P_0)$ is the algebra of regular functions for $X$. We call $R(A, P_0)$ a \textit{trinomial algebra} of the respective type (1 or 2).

We now define the finest grading of a trinomial algebra. Let $P_0^T$ denote the transpose of~$P_0$. Consider the quotient homomorphism
$$
Q \colon \mathbb{Z}^{n+m} \to K_0 := \mathbb{Z}^{n+m}/\im (P_0^T).
$$
Let $v_{ij}, v_k \in\mathbb{Z}^{n+m}$ be the standard basis vectors. Define a $K_0$-grading of $\mathbb{K}[T_{ij}, S_k]$ by putting
$$
\deg T_{ij} = Q(v_{ij}) \in K_0, \qquad \deg S_k = Q(v_k) \in K_0.
$$
One can check that the sum $\mu:=l_{i1}Q(v_{i1})+\ldots+l_{i{n_i}}Q(v_{i{n_i}})$ is the same for all monomials $T_i^{l_i}$. Therefore, all polynomials $\mathfrak g_i$ are homogeneous with respect to our grading. Hence, the ideal $\mathfrak I$ is also homogeneous and the grading of the quotient algebra $R(A, P_0)$ is well defined. We will further consider a trinomial algebra to be $K_0$-graded.
\end{constr}

\begin{example}\label{ExampleVar}
     Let $n=4$, $r=3$, $m=0$. Set
     $$
        A = \begin{pmatrix}
            0 & -1 & 1\\
            1 & -1 & 0\\
        \end{pmatrix},\qquad
     P_0 = \begin{pmatrix}
         -2 & 2 & 1 & 0\\
         -2 & 0 & 0 & 4\\
     \end{pmatrix}.
     $$
     Then $R(A, P_0) = \mathbb K[T_{01}, T_{11}, T_{12}, T_{21}]/(T_{01}^2 + T_{11}^2T_{12} + T_{21}^4)$ is a trinomial algebra of type~2.    
     The matrix~$P_0$ has Smith normal form 
     $
     \begin{pmatrix}
        1 & 0 & 0 &0\\
        0 & 2 & 0 & 0\\
     \end{pmatrix}.
     $
     Therefore $K_0  = \mathbb Z^4 / \im P_0^T\simeq \mathbb Z^2\oplus\mathbb Z_2$. The grading can be given as
     $$
     \deg T_{01} = \begin{pmatrix}
         2\\ 0\\ [1]_2\\
     \end{pmatrix},\quad\deg T_{11} = \begin{pmatrix}
     0\\ 1\\ [0]_2\end{pmatrix},\quad\deg T_{12} = \begin{pmatrix}4\\ -2\\ [0]_2\end{pmatrix},\quad \deg T_{21} = \begin{pmatrix} 1\\ 0\\ [1]_2\end{pmatrix}.
     $$
     In this case, $$\mu = \deg(T_{01}^2 + T_{11}^2T_{21} + T_{21}^4) = \begin{pmatrix}
         4\\
         0\\
         [0]_2
     \end{pmatrix}.$$
\end{example}

\begin{remark}\label{degRemark}
    Denote $e_k = \deg S_k \in K_0$. The definition of $K_0$ implies that $\langle e_1,\ldots,e_m \rangle \hm\simeq \mathbb{Z}^m$ is a direct summand in $K_0$. Hence, for every $v \in K_0$ the coefficient of $e_k$ is well defined. We denote it by~$e^k(v)$. Note that $S_k$ is the only variable with $e^k(S_k)\neq 0$. Therefore, the grading $\deg$ is non-negative for $e_1,\ldots,e_m$, i.e., a homogeneous component $R(A, P_0)_v = \{f\in R(A,P_0) \,|\, \deg f = v\,\}$ has to be zero whenever $e^k(v)<0$ for some $k$. Furthermore, if $e^k(v)=d>0$, then $S_k^d$ divides any element of $R(A, P_0)_v$.
\end{remark}

\begin{remark}\label{Type1Remark}
    Let $R(A, P_0)$ be of type~1. Then
    $$
    K_0\simeq\bigoplus_{i=1}^r \mathbb Z^{n_i}/\langle(l_{i1},\ldots,l_{in_i})\rangle \oplus \mathbb Z^m.
    $$    
\end{remark}

\begin{lemma}\label{KerLem1}
    Let $R(A, P_0)$ be of type 1. Then every $K_0$-homogeneous element $h\in R(A, P_0)$ has the form
    $$
    h = F(T_1^{l_1})\Prod T_{ij}^{b_{ij}}\Prod S_k^{d_k} 
    $$
    with $b_{ij}, d_k\in\mathbb Z_{\ge 0}$ and $F$ being a polynomial in one variable.
\end{lemma}

\begin{proof}
    One can write any monomial $f$ in variables $T_{ij}, S_k$ in the form
    $$
    f = \lambda\Prod_i (T_i^{l_i})^{q_i}\Prod\limits_{j} T_{ij}^{b_{ij}}\Prod\limits_{k} S_k^{d_k},
    $$
    where $\lambda\in\mathbb K, q_i, b_{ij}, d_k\in\mathbb Z_{\ge0}$, and for any $i$ there exists $j$ with $ b_{ij}<l_{ij}$. By Remark~\ref{Type1Remark}, such integers $b_{ij}, d_k$ correspond to precisely one element of $K_0$, and thus can be recovered from $\deg f$.

    Any $K_0$-homogeneous element $h\in R(A, P_0)$ is a sum of monomials with the same $b_{ij}, d_k$. Hence, 
    $$
    h = H(T_1^{l_1},\ldots, T_r^{l_r})\Prod\limits_{i,j} T_{ij}^{b_{ij}}\Prod\limits_{k} S_k^{d_k}
    $$
    for a polynomial $H$ in $r$ variables. According to equations $\mathfrak g_i=0$, monomials $T_i^{l_i}$ are actually polynomials in~$T_1^{l_1}$. As a consequence, $h$ has the required form.
\end{proof}

The following lemma is proved in \cite[Lemma 4]{GZ} in the case of a trinomial hypersurface of type~$2$ without variables $S_k$.

\begin{lemma}\label{KerLem}
    Let $R(A, P_0)$ be of type 2. Then every $K_0$-homogeneous element $h\in R(A, P_0)$ has a form
    $$
    h = F(T_1^{l_1}, T_2^{l_2})\Prod T_{ij}^{b_{ij}}\Prod S_k^{d_k} 
    $$
    with $d_k, b_{ij}\in\mathbb Z_{\ge 0}$ and $F$ being a homogeneous polynomial in two variables.
\end{lemma}

\begin{proof}
    It follows from the proof of \cite[Proposition 3.5]{AHHL} that $h = F(T_1^{l_1}, T_2^{l_2})\Tilde{h}$ for a homogeneous polynomial $F$ and some $\Tilde{h}\in R(A, P_0)_{\Tilde\omega}$ with $\dim R(A, P_0)_{\Tilde{\omega}} = 1$. Let $\Tilde h = \sum\limits_{q=1}^m h_q$ where $h_q\in R(A, P_0)_{\Tilde\omega}$ are linearly independent monomials in $T_{ij}$ and~$S_k$. As $\dim R(A,P_0)_{\Tilde\omega}=1$, we see that~$\Tilde h$ has to be a monomial.
\end{proof}

The following lemma asserts that the group $K_0$ provides the finest grading of a trinomial algebra. Although it is proved in~\cite[Lemma 1]{Z} for a hypersurface of type 2, the proof in general case differs only in technical details.

\begin{lemma}\label{ZLemma}
Let $\deg$ be the $K_0$-grading on $R(A, P_0)$ defined in Construction~\ref{ConstrR}, and $\deggg$ an arbitrary grading on $R(A, P_0)$ by an abelian group $\widehat{K}$ such that all $T_{ij}$ and $S_k$ are homogeneous. Then $\deggg = \psi \circ \deg$ for some homomorphism ${\psi \colon K_0 \to \widehat{K}}$. 
In particular, any $K_0$-homogeneous locally nilpotent derivations is also $\widehat{K}$-homogeneous.
\end{lemma}



Let us recall some notions. Consider an abelian group $G$ and a $G$-graded $\mathbb K$-domain $R$. A non-zero non-unit homogeneous element $p\in R$ is said to be $G$-\textit{prime} if for any homogeneous elements $a, b\in R$ the condition $p\, | \, ab$ implies $p\, | \,a$ or $p\, | \,b$. A $G$-grading of a $\mathbb K$-domain $R$ is $G$-\textit{factorial} if every non-zero non-unit $G$-homogeneous element $a\in R$ is a product of $G$-primes. Evidently, such a decomposition is unique up to order and units.

In~\cite[Theorem 1.2]{HW}, the following properties of trinomial algebras are developed.

\begin{stat}\label{PropHausen}
    Consider a $K_0$-graded trinomial algebra $R(A, P_0)$.
    \begin{enumerate} 
    \item[\textrm{(i)}] $R(A, P_0)$ is a $\mathbb K$-domain.
    \item[(ii)] The $K_0$-grading of $R(A, P_0)$ is factorial and $R(A, P_0)$ has only constant invertible homogeneous elements.
    \item[(iii)] The variables $T_{ij}$ and $S_k$ are pairwise nonassociated $K_0$-prime generators for $R(A, P_0)$.
    \end{enumerate} 
\end{stat}

\section{Elementary derivations}
Here we define some $K_0$-homogeneous locally nilpotent derivations on $R(A, P_0)$ that we call elementary. The following lemma provides some basic properties of locally nilpotent derivations. The proof can be found in \cite[Principles 1, 5, 7, Corollary 1.23]{Freu}.

\begin{lemma}\label{BasicLem}
    Let $R$ be a $\mathbb K$-domain and $\delta$ a locally nilpotent derivation on $R$. Suppose $f, g \in R.$
    \begin{enumerate} 
    \item[\textrm{(i)}] $f\delta$ is a derivation. It is locally nilpotent if and only if $f\in\Ker\delta$.
    \item[(ii)] $\Ker\delta$ is factorially closed, i.e. $f, g\in\Ker\delta$ whenever $fg \in \Ker\delta$.
    \item[(iii)] If $f\, | \, \delta(f)$, then $\delta(f)=0$.
    \item[(iv)] If $f\, | \, \delta(g)$ and $g\, | \, \delta(f)$, then $\delta(f)=0$ or $\delta(g)=0$.
    \end{enumerate} 
\end{lemma}

\begin{constr}\label{ConstrDer}
    Let us define some natural types of locally nilpotent derivations on a trinomial algebra. Some of them were introduced in~\cite[Construction 4.2]{AHHL}.
    \begin{enumerate}
        \item[0.] Fix $p$ with $1\le p\le m$ and define as usual
        \begin{equation*}
            \begin{split}
                {\partial\over\partial S_p}(T_{ij}) &= 0 \quad \text{for all } (i, j);\\
                 {\partial\over\partial S_p}(S_k) &= \begin{cases}
                    1, &\text{$k=p$,}\\
                    0, &\text{$k \neq p$.}
                \end{cases}
            \end{split}
        \end{equation*}
    
        \item[1.] Consider a trinomial algebra of type 1. Suppose that there exists a tuple $C \hm= (c_1, \ldots, c_r)$ with $1\le c_i \le n_i$ and $l_{ic_i}=1$ for all $i$ except at most one. Then we define
         \begin{equation*}
             \begin{split}
                 \delta_C(T_{ij}) &= \begin{cases}
                 \Prod\limits_{k\neq i}\frac {\partial T_k^{l_k}}{\partial T_{kc_k}},&\text{$j=c_i$,}\\
                 0, &\text{$j \neq c_i$;}
                 \end{cases}\\
                 \delta_C(S_k) &= 0\quad \text{for all } k.
             \end{split}
         \end{equation*}

     \item[2.] Consider a trinomial algebra of type 2. Fix a vector $\beta \in \mathbb K^{r+1}$ lying  in the row space of the matrix $A \hm= (a_0,\ldots, a_r)$. Notice that if $\beta\neq 0$ then either $\beta_i \neq 0$ for all $i$ or $\beta_{i_0} = 0$ for exactly one index~$i_0$. Thus, we have two cases.
         \begin{enumerate}
            \item[2.1.] Suppose that $\beta_i \neq 0$ for all $i$ and there exists a tuple $C=(c_0,\ldots,c_r)$ with $1\le c_i\le n_i$ and $l_{ic_i}=1$ for all $i$ except at most one. Then define
                \begin{equation*}
                    \begin{split}
                        \delta_{C,\beta}(T_{ij}) &=
                            \begin{cases}
                            \beta_i\Prod\limits_{k\neq i}{\partial T_k^{l_k}\over\partial T_{kc_k}}, &\text{$j = c_i$,}\\
                            0, &\text{$j \neq c_i$;}
                            \end{cases}\\
                        \delta_{C,\beta}(S_k) &= 0 \quad\text{for all $k$.}
                    \end{split}
                \end{equation*}
            \item[2.2.] Suppose that $\beta_{i_0}=0$ for exactly one $i_0$ and there exists a tuple $C=(c_0,\ldots,c_r)$ with $1\le c_i\le n_i$ and $l_{ic_i}=1$ for all $i\neq i_0$ except at most one. Then define
                \begin{equation*}
                    \begin{split}
                        \delta_{C,\beta}(T_{ij}) &= 
                            \begin{cases}
                                \beta_i\Prod\limits_{k\neq i, i_0}{\partial T_k^{l_k}\over\partial T_{kc_k}}, &\text{$j = c_i$}\\
                                0, &\text{$j \neq c_i$;}
                            \end{cases}\\
                        \delta_{C,\beta}(S_k) &= 0 \quad\text{for all $k$.}
                    \end{split}
                \end{equation*}
         \end{enumerate}
    \end{enumerate}
     
    The assignments ${\partial\over\partial S_p}$, $\delta_C$ and $\delta_{C,\beta}$ are defined for the generators of the algebra $\mathbb K [T_{ij}, S_k]$ and can be extended to derivations by the Leibniz rule and linearity. One can check that ${\partial\over\partial S_p}(\mathfrak I) = \delta_C(\mathfrak I) = \delta_{C,\beta}(\mathfrak I) = 0$. Consequently, our derivations, if they exist, are well defined for the quotient algebra~$R(A, P_0)$. The existence of these derivations is equivalent to the condition $m>0$ for item~$0$ and some conditions on exponents $l_{ij}$ for items 1 and 2.
\end{constr}

\begin{example}
Apply Construction~\ref{ConstrDer} to $R(A, P_0) = \mathbb K[T_{01}, T_{11}, T_{12}, T_{21}]/(T_{01}^2 \hm+ T_{11}^2T_{12} + T_{21}^4)$ as in Example~\ref{ExampleVar}. Clearly, the vector $\beta\in\mathbb K^3$ lies in the row space of the matrix $A$ exactly when $\beta_0+\beta_1+\beta_2 = 0$. As $m=0$ and $T_{12}$ is the only variable with $l_{ij}=1$, only the case 2.2 with $C = (1, 2, 1)$ is possible. We have two possibilities for $i_0$:
     \begin{enumerate}
         \item [a)] $i_0 = 0.$ Then $\beta = (0, \beta_1, -\beta_1)$ and we have
         $$
         \delta_{C,\beta}(T_{01}) =  0,\quad \delta_{C,\beta}(T_{11}) = 0,\quad \delta_{C,\beta}(T_{12}) = 4\beta_1 T_{21}^3,\quad \delta_{C,\beta}(T_{21}) = -\beta_1 T_{11}^2.
         $$
         \item[b)] $i_0=2$. Then $\beta = (\beta_0, -\beta_0, 0)$ and we have
         $$
         \delta_{C,\beta}(T_{01}) = \beta_0 T_{11}^2,\quad \delta_{C,\beta}(T_{11}) = 0,\quad \delta_{C,\beta}(T_{12}) = -2\beta_0 T_{01},\quad \delta_{C,\beta}(T_{21}) = 0.
         $$
     \end{enumerate}
\end{example}

The below definition follows~\cite[Definition~1]{GZ}.

\begin{defn}
    We call a derivation $\widehat{\delta} \colon R(A, P_0) \to R(A, P_0)$ \textit{elementary} if either $\widehat{\delta}=0$ or $\widehat{\delta} = h\delta$, where $\delta$ is defined in Construction~\ref{ConstrDer}, and $h\in \Ker{\delta}$ is a homogeneous element.
\end{defn}

\begin{lemma}
    Every elementary derivation is locally nilpotent and $K_0$-homogeneous.
\end{lemma}

\begin{proof}
    Let $\widehat{\delta} = h\delta$ with $\delta$ defined in Construction~\ref{ConstrDer} and $h\in\Ker\delta$ be a homogeneous element. It suffices to check that $\delta$ is locally nilpotent and $K_0$-homogeneous. Then $\widehat{\delta}$ is locally nilpotent by Lemma~\ref{BasicLem}. Clearly, $\widehat{\delta}$ is also $K_0$-homogeneous (of degree $\deg\widehat{\delta} = \deg h + \deg\delta$).
    
    One can check that $\delta$ is nilpotent for generators $T_{ij}, S_k$ and thus locally nilpotent for $R(A, P_0)$. It also follows from considering the generators that $\delta$ is $K_0$-homogeneous of degree:
    \begin{align*}
            \deg{\partial\over\partial S_p} &=& - \deg S_p &&\text{for item 0;}\\
            \deg\delta_C &=& - \sum\limits_k{\deg T_{k{c_k}}} &&\text{for item 1;}\\
            \deg\delta_{C,\beta} &=& r\mu - \sum\limits_k\deg T_{k{c_k}} &&\text{for item 2.1;}\\
            \deg\delta_{C,\beta} &=& (r-1)\mu - \sum\limits_{k\neq i_0}\deg T_{k{c_k}} &&\text{for item 2.2.}
        \end{align*}
    Here $\mu = \deg T_0^{l_0} = \ldots = \deg T_r^{l_r}$ is the degree of polynomials $\mathfrak g_i$.
\end{proof}

The converse statement is the main result of the paper.

\begin{theorem}\label{MainTheorem}
    Any $K_0$-homogeneous locally nilpotent derivation of a trinomial algebra is elementary.
\end{theorem}

This theorem is proved in Section~\ref{mainproof}; see Theorem~\ref{Type1Theorem} for trinomial algebras of type 1 and Theorem~\ref{Type2Theorem} for type 2.

\section{Some auxiliary results}


By definition, any trinomial variety $X\subset \mathbb A^{n+m}$ admits a decomposition $X= Y\times \mathbb A^m$, where $Y\subset\mathbb A^n$ is the trinomial variety given by the same equations $\mathfrak g_i=0$ in the affine space with coordinates $T_{ij}$, and $\mathbb A^m$ is the affine space with coordinates $S_1,\dots,S_m$. Thus, $\mathbb K[X] = \mathbb K[Y] \otimes \mathbb K[\mathbb A^m]$. Evidently, both $\mathbb K[Y]$ and $\mathbb K[\mathbb A^m]$ are $K_0$-graded subalgebras of $\mathbb K[X]$. Moreover, weight monoids of two these subalgebras have zero intersection. The following lemma shows that any $K_0$-homogeneous LND for $\mathbb K[X]$ comes from either a homogeneous derivation of~$\mathbb K[Y]$ or a homogeneous derivation of~$\mathbb K[\mathbb A^m]$.

\begin{lemma}\label{ST-AlterLem}
    Let $\delta$ be $K_0$-homogeneous locally nilpotent derivation for a trinomial algebra $R(A, P_0)$. Then either $\delta(S_k)=0$ for all $k$ or $\delta(T_{ij})=0$ for all $i,j$.
\end{lemma}

\begin{proof}
Let $e_k \hm= \deg S_k$ and $e^k(v)$ denotes the coefficient of $e_k$ for $v\in K_0$. According to Remark~\ref{degRemark}, for arbitrary homogeneous element $f$, the condition $e^k(\deg f) > 0$ implies $S_k \,|\,f$, and $e^k(\deg f)<0$ implies $f=0$. 
     
     We have two possibilities:
     \begin{enumerate}
    \item[a)] $e^k(\deg\delta) \ge 0$ for all $k.$ Then $e^k(\deg\delta(S_k)) = e^k(\deg\delta + \deg S_k) = e^k(\deg\delta) + 1 > 0$. Consequently, $S_k\, | \, \delta(S_k)$ and Lemma~\ref{BasicLem} asserts $\delta(S_k)=0$.
    \item[b)] $e^k(\deg\delta) < 0$ for some $k$. In this case, we have $e^k(\deg\delta(T_{ij})) = e^k(\deg\delta) + e^k(\deg T_{ij}) \hm= e^k(\deg\delta) < 0$, hence, $\delta(T_{ij})=0.$
     \end{enumerate}

     The assertion follows.
\end{proof} 

\begin{defn}
    Let $\delta$ be a $K_0$-homogeneous locally nilpotent derivation of a trinomial algebra. Taking into account the above lemma, we call $\delta$ \textit{proper} if $\delta\neq 0$ but $\delta(S_k)=0$ for all $k$. In the other case, when $\delta(T_{ij})=0$ for all $i,j$, we call $\delta$ \textit{non-proper}.
\end{defn}
    In other words, proper derivations on~$\mathbb K[X]$ come from nonzero homogeneous derivations on $\mathbb K[Y]$, and non-proper ones come of homogeneous derivations on $\mathbb K[\mathbb A^m]$.
    
\begin{lemma}\label{S-lemma}
    A non-proper $K_0$-homogeneous derivation is elementary.
\end{lemma}

\begin{proof}
Let $\delta$ be a non-proper derivation. The case $\delta=0$ is evident, so we can assume $\delta(S_{k_0})\neq 0$ for some $k_0$. As $\delta$ is homogeneous,
$\delta(S_{k_0})$ is a homogeneous element of $R(A, P_0)$. Decompose it into the sum of monomials. By Remark~\ref{degRemark}, the maximal power of $S_{k_0}$ that divides a monomial is the same for all monomials in $\delta(S_{k_0})$. Since $\delta(S_{k_0})\neq 0$, Lemma~\ref{BasicLem} implies $S_{k_0} \nmid\delta(S_{k_0})$, and therefore, $e^{k_0}(\deg\delta(S_{k_0}))=0$. Then $e^{k_0}(\deg\delta) = -1$, which forces $\delta(S_k)$ to be zero for $k\neq k_0$. As a result, $\delta = \delta(S_{k_0}){\partial\over\partial S_{k_0}}$ is an elementary derivation.
\end{proof}

\begin{defn}
Let $\delta$ be a locally nilpotent derivation for $R(A, P_0)$. We call the variable $T_{ij}$ \textit{non-kernel} (with respect to $\delta$) if $T_{ij}\notin \Ker{\delta}$. In this case, we also call $l_{ij}$ a \textit{non-kernel exponent}.
\end{defn}

\begin{lemma}\label{nonKerLem}
Let $\delta$ be $K_0$-homogeneous locally nilpotent derivation on $R(A, P_0)$. Then every monomial~$T^{l_i}_i$ has at most one non-kernel variable.
\end{lemma}

\begin{proof}

The proof is by reductio ad absurdum. Suppose that there exists a monomial with two non-kernel variables. We can assume that these variables are $T_{11}$ and $T_{12}$ in $T^{l_1}_1$. Define $\mathbb{Z}$-grading $\degg$ for $\mathbb{K}[T_{ij}, S_k]$ as follows:
$$\degg{T_{11}}=l_{12},\qquad \degg{T_{12}}=-l_{11},$$ $$\degg{T_{ij}}=\degg{S_k}=0 \quad \text{ for any other } (i, j) \mbox{ and any } k.$$
Evidently, all polynomials $\mathfrak g_i$ are homogeneous (of degree $0$) with respect to $\degg$. Thus, $R(A, P_0)$ is $\degg$-graded, and Lemma~\ref{ZLemma} forces $\delta$ to be homogeneous with respect to $\degg$. Only the following two cases are possible:

\begin{enumerate}
    \item[a)] If $\degg{\delta} \ge 0$, then $\degg{\delta(T_{11})}=\degg{\delta} + \degg{T_{11}} > 0$. Hence, $T_{11}\,|\,\delta(T_{11})$ because $T_{11}$ is the only variable with positive degree with respect to $\degg$. Lemma~\ref{BasicLem} implies that $\delta(T_{11})=0$.
    \item[b)] If $\degg{\delta} \le 0$, then $\degg{\delta(T_{12})}=\degg{\delta} + \degg{T_{12}} < 0$, and similarly we have $T_{12}\,|\,\delta(T_{12})$. It follows that $\delta(T_{12})=0$. 
\end{enumerate}
A contradiction in both cases.
\end{proof}

\begin{lemma}\label{nonKerDegLem}
    Assume that two different variables $T_{pc_p}$ and $T_{qc_q}$ appear in $T_p^{l_p}$ and $T_q^{l_q}$ with exponents $l_{pc_p}>1$ and $l_{qc_q}>1$. Then for every $K_0$-homogeneous LND~$\delta$ one has $\delta(T_{pc_p})=0$ or $\delta(T_{qc_q})=0$.
\end{lemma}

\begin{proof}
    Suppose that $\delta(T_{ic_i})\neq0$ and $\delta(T_{jc_j})\neq0$ for some $\delta$. Define the grading of $\mathbb K[\,T_{ij}, S_k\,]$ by the group $\mathbb Z_{l_{pc_p}}=\{[0]_{l_{pc_p}},\ldots,[l_{pc_p}-1]_{l_{pc_p}}\}$ by setting:
    $$
           \degg(T_{pc_p}) = [1]_{l_{pc_p}}, \qquad \degg(T_{ij}) = [0]_{l_{pc_p}} \quad\mbox{for all other $(i,j)$},
    $$
    $$
    \degg(S_k) = [0]_{l_{pc_p}}\quad\mbox{for all $k$}.
    $$
    Clearly, all $\mathfrak g_i$ are $\degg$-homogeneous polynomials (of degree 0). Hence, $R(A, P_0)$ is $\degg$-graded. It follows from Lemma~\ref{ZLemma} that derivation $\delta$ is $\degg$-homogeneous.

    Assume that $\degg\delta(T_{pc_p})\neq[0]_{l_{pc_p}}$. Then $T_{pc_p}\,|\,\delta(T_{pc_p})$ since $T_{pc_p}$ is the only variable with non-zero degree with respect to $\degg$. But this implies $\delta(T_{pc_p})=0$ by Lemma~\ref{BasicLem}. This contradiction proves that
    $$
    \degg\delta(T_{pc_p})=[0]_{l_{pc_p}}.
    $$
    Therefore, 
    $$
    \degg\delta = \degg\delta(T_{pc_p}) - \degg T_{pc_p} = [-1]_{l_{pc_p}}\neq [0]_{l_{pc_p}}.
    $$
    In consequence,
    $$
    \degg\delta(T_{qc_q}) = \degg\delta + \degg T_{qc_q}= \degg\delta \neq [0]_{l_{pc_p}}.
    $$
    Since $T_{pc_p}$ is the only variable with nonzero degree, we obtain
    $$
    T_{pc_p}\,|\,\delta(T_{qc_q}).
    $$
    By considering the similar graging by the group $\mathbb Z_{l_{qc_q}}$, we have:
    $$
    T_{qc_q}\,|\,\delta(T_{pc_p}).
    $$
    By Lemma~\ref{BasicLem}, either $\delta(T_{pc_p})=0$ or $\delta(T_{qc_q})=0$.
\end{proof}

\section{The proof of main result}\label{mainproof}
In this section, we prove Theorem~\ref{MainTheorem} separately for trinomial algebras of type 1 (see Theorem~\ref{Type1Theorem}) and for type 2 (see Theorem~\ref{Type2Theorem}).

\begin{theorem}\label{Type1Theorem}
Let $R(A, P_0)$ be of type 1 and $\delta$ a $K_0$-homogeneous locally nilpotent derivation of $R(A, P_0)$. Then $\delta$ is elementary.
\end{theorem}

\begin{proof} 
By Lemma~\ref{S-lemma}, we can assume that $\delta$ is proper (in particular, $\delta\neq0$). As
$$
0=\delta(\mathfrak g_i) = \delta(T_i^{l_i})-\delta(T_{i+1}^{l_{i+1}})
$$
for all $i\in I$, one has
\begin{equation}\label{fdef}
\delta(T_1^{l_1}) = \ldots = \delta(T_r^{l_r}) =: f\in R(A, P_0). 
\end{equation}
According to Lemma~\ref{nonKerLem}, every monomial $T_i^{l_i}$ has at most one non-kernel variable. Then, by the Leibniz rule,
\begin{equation}\label{delta}
\delta(T_i^{l_i})={\partial T_i^{l_i}\over \partial T_{ic_i}}\delta(T_{ic_i})\qquad\mbox{for some $c_i$}.
\end{equation}
As $\delta\neq 0$, formulae~(\ref{fdef}) and (\ref{delta}) immediately imply that every monomial $T_i^{l_i}$ has exactly one non-kernel variable~$T_{ic_i}$. Hence,
\begin{equation}\label{division}
{\partial T_i^{l_i}\over\partial T_{ic_i}} \quad\text{divides}\quad\,f\quad\text{for all $i$}.
\end{equation}
By Proposition~\ref{PropHausen}, the algebra $R(A,P_0)$ is $K_0$-factorial, and variables $T_{ij}$ are $K_0$-primes. Then, by (\ref{division}), 
$$
f = g \Prod_{k=1}^r {\partial T_k^{l_k}\over\partial T_{kc_k}}
$$
with $g\in R(A,P_0)$ a $K_0$-homogeneous element. Now it follows by~(\ref{fdef}) and (\ref{delta}) that
$$
\delta(T_{ic_i}) = g \Prod_{k\neq i} {\partial T_k^{l_k}\over\partial T_{kc_k}} \qquad\text{for all $i$}.
$$
Lemma~\ref{nonKerDegLem} asserts that $l_{ic_i}=1$ for all $i$ with exception of at most one. Thus, $\delta = g\delta_C$ with $C=(c_1,\ldots,c_r)$ 
and $\delta_C$ defined in Construction~\ref{ConstrDer}, item 1. Lemma~\ref{BasicLem} implies that $g\in\Ker\delta_C$. To sum up, $\delta$ is elementary.
\end{proof}

Now we should focus on trinomial algebras of type 2. As for type 1, it suffices to threat only the case of proper derivation.

\begin{remark}\label{DimRemark}
    Let $R(A, P_0)$ be of type~2. By equation $\mathfrak g_i=0$, any monomial $T_i^{l_i}$ is a linear combination of two any fixed monomials $T_p^{l_p}$ and $T_q^{l_q}$. Evidently, for any derivation $\delta$, one has 
    $$
    \Lin\left\{\delta(T_i^{l_i}) \,|\, i=0,\ldots,r\right\} = \Lin\left\{\delta(T_p^{l_p}), \,\delta(T_q^{l_q})\right\} \qquad\mbox{for any $p\neq q$}.
    $$
\end{remark}

\begin{stat}\label{DimStat}
     Let $R(A, P_0)$ be of type~2, $\delta$ a proper $K_0$-homogeneous locally nilpotent derivarion of $R(A, P_0)$. If 
    $$
    \dim \Lin\left\{ \delta(T_i^{l_i}) \,|\, i=0,\ldots,r\right\} =1,
    $$
    then $\delta$ is elementary.
\end{stat}

\begin{proof}
    According to Lemma~\ref{nonKerLem}, any monomial $T_i^{l_i}$ has at most one non-kernel variable. We have the following two cases.

        \textit{Case a.} Any monomial $T_i^{l_i}$ has exactly one non-kernel variable. For all $i$, define $c_i$ by the condition $\delta(T_{ic_i})\neq 0$. The Leibniz rule implies that
        $$
        \delta(T_i^{l_i}) = {\partial T_i^{l_i}\over\partial T_{ic_i}} \delta(T_{ic_i}).
        $$
        Since all $\delta(T_i^{l_i})$ lie in an one-dimensional subspace, one has
        $$
        {\partial T_i^{l_i}\over\partial T_{ic_i}} \delta(T_{ic_i}) = \beta_i f
        $$
        for some
        $\beta_i\in\mathbb K^{\times}$, $f\in R(A, P_0)$ and all $i$. Thus
        $$
        {\partial T_i^{l_i}\over\partial T_{ic_i}}\quad\text{divides}\quad f\quad\text{for any $i$.}
        $$
         As all $T_{ij}$ are $K_0$-prime and $R(A, P_0)$ is $K_0$-factorial, we have:
        $$
        f = g \Prod_k {\partial T_k^{l_k}\over\partial T_{kc_k}} = \beta_i^{-1}{\partial T_i^{l_i}\over\partial T_{ic_i}}\delta(T_{ic_i})
        $$
        for some $g\in R(A,P_0)$ and all $i$. Then for all $i$
        $$
        \delta(T_{ic_i}) = \beta_i g \Prod_{k\neq i} {\partial T_k^{l_k}\over\partial T_{kc_k}}.
        $$

        Denote $C = (c_0, \ldots, c_r)$, $\beta = (\beta_0,\ldots,\beta_r)$. We claim that the derivation $\delta_{C,\beta}$ is well defined, namely, the tuples $C,\beta$ satisfy the conditions in item~2.1 of Construction~\ref{ConstrDer}. Indeed, by Lemma~\ref{nonKerDegLem}, we have $l_{ic_i}=1$ or $l_{jc_j}=1$ for any pair $i\neq j$. Furthermore, the following calculation shows that the vector $\beta$ lies in the row space of the matrix~$A$:
        \begin{multline*}
        0 = \delta(\mathfrak g_i)
        =\delta\left( \det \begin{pmatrix}
        T_i^{l_i} & T_{i+1}^{l_{i+1}} & T_{i+2}^{l_{i+2}}\\
        a_i & a_{i+1} & a_{i+2}
        \end{pmatrix} \right) =\\
        =\det \begin{pmatrix}
        \delta(T_i^{l_i}) & \delta(T_{i+1}^{l_{i+1}}) & \delta(T_{i+2}^{l_{i+2}})\\
        a_i & a_{i+1} & a_{i+2}
        \end{pmatrix} =
        \det \begin{pmatrix}
        \beta_i & \beta_{i+1} & \beta_{i+2}\\
        a_i & a_{i+1} & a_{i+2}
        \end{pmatrix} f.
        \end{multline*}
        
        Hence, the derivation $\delta_{C,\beta}$ is well defined. By Lemma~\ref{BasicLem}, $g\delta_{C,\beta}$ is a derivation, and it is locally nilpotent if and only if $g\in\Ker\delta_{C,\beta}$. Obviously, the assignments
        $\delta$ and $g\delta_{C,\beta}$ coincide for the generators~$T_{ij},S_k$, and consequently $\delta=g\delta_{C,\beta}$. 
        Thus, $\delta$ is elementary.

    \textit{Case b.} There exists a monomial $T_{i_0}^{l_{i_0}}$ that has no non-kernel variable. Then $i_0$ is defined uniquely. Indeed, if another monomial $T_{i_1}^{l_{i_1}}$ has no non-kelnel variable, by Remark~\ref{DimRemark} one has
    $$
    \Lin\left\{\delta(T_i^{l_i}) \,|\, i=0,\ldots,r\right\} = \Lin\left\{\delta(T_{i_0}^{l_{i_0}}), \,\delta(T_{i_1}^{l_{i_1}})\right\} = 0.
    $$
    This is impossible because $\delta\neq0$ (since $\delta$ is proper by assumption).

    Hence, $\delta(T_{i_0 j})=0$ for all $j$. And for all $i\neq i_0$, the argument in the \textit{case a} is workable. Namely, for any $i\neq i_0$ there exists $c_i$ with
    $$
    \delta(T_{ij}) = 
    \begin{cases}
        \beta_i g \Prod\limits_{k\neq i, i_0} {\partial T_k^{l_k}\over\partial T_{kc_k}}, &j=c_i,\\
        0, &j\neq c_i
    \end{cases}
    $$
    for some $\beta_i\in\mathbb{K^\times}$ and $g\in R(A, P_0)$. 
    
    Set $c_{i_0}$ arbitrary, $\beta_{i_0}=0$. One can check, similar to the \textit{case a}, that the tuples $C=(c_0,\dots,c_r),$ $\beta = (\beta_1,\ldots,\beta_r)$ satisfy the conditions in Construction~\ref{ConstrDer}, item~2.2; i.e., the derivation $\delta_{C,\beta}$ is well defined. As in the \textit{case a}, one obtains $\delta = g\delta_{C,\beta}$ and $g\in\Ker\delta_{C,\beta}$, so $\delta$ is elementary. 
\end{proof}

    The following lemma is proved in~\cite[Lemma 9]{GZ}.

\begin{lemma}\label{G-Z_Lem}
    Let $\partial$ be a $K_0$-homogeneous locally nilpotent derivation on the trinomial variety
    $$
    \{\,x^k+y+z=0\,\} \subset \mathbb K^3.
    $$
    Then 
    $$
    \dim\Lin\{\partial(x^k),\,\partial(y),\,\partial(z)\} = 1.
    $$
\end{lemma}

Given a $G$-graded $\mathbb K$-domain $R$, the subfield of degree zero fraction in the fraction field is defined as
$$
\Quot(R)_0 = \left\{ \left. {f\over g}\,\right|\,f,g\in R \text{ are $G$-homogeneous, } \quad g\neq 0,\quad \deg f = \deg g  \right\}.
$$

The following lemma is proved in~\cite[Proposition 3.4]{AHHL}.

\begin{lemma}\label{QuotLem}
    Let $R(A, P_0)$ be a trinomial algebra of type~2. Fix any $0 \le i,j\le r, i\neq j$. Then
    $$
    \Quot(R(A, P_0))_0 = \mathbb K \left({T_i^{l_i}\over T_j^{l_j}}\right).
    $$
\end{lemma}

\begin{stat}\label{forDimStat}
    Let $\delta$ be a proper $K_0$-homogeneous locally nilpotent derivation of a trinomial algebra $R(A, P_0)$ of type~2. Then
    $$
   \dim \Lin\left\{ \delta(T_i^{l_i}) \,|\, i=0,\ldots,r\right\} =1.
    $$
\end{stat}
\begin{proof}
    Lemma~\ref{nonKerLem} claims that $\delta$ has at mot one non-kernel variable in every monomial $T_i^{l_i}$. If a monomial $T_p^{l_p}$ has no non-kernel variable, it follows by Remark~\ref{DimRemark} that for any $q\neq p$ one has
    $$
    \dim \Lin\left\{ \delta(T_i^{l_i}) \,|\, i=0,\ldots,r\right\} = \dim\Lin\left\{ \delta(T_p^{l_p}),\, \delta(T_q^{l_q})\right\} = 1.
    $$
    
    It suffices to threat the case when every monomial $T_i^{l_i}$ has exactly one non-kernel variable. According to Lemma~\ref{nonKerDegLem}, for non-kernel exponents we have $l_{ic_i}=1$ except at most one $i=i_0$. Let us suppose      
    \begin{equation*}    
    T_{i1}\notin\Ker\delta \quad\mbox{for all $i$}\quad\qquad \mbox{and}\quad\qquad l_{i1}=1 \quad\mbox{for all $i>0$}.
    \end{equation*}
    In other words, any non-kernel variable is the first in its monomial, and non-unit non-kernel exponent is allowed only in the zeroth monomial $T_0^{l_0}$.
    
    Let $\mathbb L$ be the algebraic closure of the field obtained from $\mathbb K$ by adjunction of all kernel variables:
    $$
    \mathbb L = \overline{\mathbb K(T_{ij}, S_k\,|\, 0\le i \le r,\, 2\le j\le n_i,\, 1\le k\le m)} \subseteq \overline{\Quot R(A, P_0)}.
    $$
    Denote also the polynomial algebra in $(r+1)$ variables by
    $$
    \mathbb L[T_{i1}] := \mathbb L[T_{01},\ldots,T_{r1}].
    $$
    Consider the variety $V\subseteq\mathbb L^{r+1}$ defined by the ideal generated by $\mathfrak g_i,\, i\in I$. Evidently this ideal is $\mathbb L \mathfrak I$, where $\mathfrak I$ is the ideal of the original trinomial variety, and $\mathbb L[V] = \mathbb L[T_{i1}]/\mathbb L \mathfrak I$. We have the embedding
    $$
    R(A, P_0) = \mathbb K[T_{ij}, S_k]/\mathfrak I \hookrightarrow\mathbb L[T_{i1}]/\mathbb L \mathfrak I = \mathbb L[V],
    $$
    induced by the natural embedding $\mathbb K[T_{ij}, S_k]\hookrightarrow\mathbb L[T_{i1}]$. As $\delta(T_{ij})=\delta(S_k)=0$ for $j\neq 1$, we can consider $\delta$ to be a locally nilpotent derivation for $ \mathbb L[T_{i1}]/\mathbb L \mathfrak I$ by setting $\delta(\mathbb L)=0$.
    By the equations $\mathfrak g_i=0$, all $T_{i1}^{l_{i1}}$ are linear combinations of $T_{01}^{l_{01}}, \,T_{11},\,T_{21}$ in $ \mathbb L[T_{i1}]/\mathbb L \mathfrak I$. 
    
    Let $\mathfrak g_0 = aT_{01}^{l_{01}}+bT_{11}+cT_{21}$ with $a,b,c\in\mathbb L$. The assignment 
    $$
    T_{01}\mapsto a^{-l_{01}^{-1}}x,\quad T_{11}\mapsto {b^{-1}y},\quad T_{21}\mapsto {c^{-1}z}
    $$
    gives an isomorphism of varieties $V\simeq \{\,x^{l_{01}} + y + z=0\,\} \subseteq \mathbb L^3$, which transforms $\delta$ into a locally nilpotent derivation $\partial$ on the variety $\{\,x^{l_{01}} + y + z=0\,\}$. Lemma~\ref{G-Z_Lem} asserts
    $$
    \dim_{\mathbb L}\Lin_{\mathbb L}\{\partial(x^k),\,\partial(y),\,\partial(z)\}=1.
    $$
    Hence, 
    $$
    \dim_{\mathbb L}\Lin_{\mathbb L}\left\{ \delta(T_i^{l_i}) \,|\, i=0,\ldots,r\right\} =1.
    $$
    Fix any pair $p\neq q$. The argument above shows that there exists $\lambda\in\mathbb L$ with
    $$
        \delta(T_p^{l_p}) = \lambda\delta(T_q^{l_q}) \quad\text{ in $\mathbb L[T_{i1}]/\mathbb L \mathfrak I$}.
    $$
    This implies that   
    $$
        {\delta(T_p^{l_p})\over \delta(T_q^{l_q})} = \lambda \quad\text{ in $\Quot(\mathbb L[T_{i1}]/\mathbb L \mathfrak I)$}.
    $$
    Since $\deg T_p^{l_p} = \deg T_q^{l_q}$, one has $\deg\delta(T_p^{l_p}) = \deg\delta(T_q^{l_q})$, and consequently
    $$
    \lambda = {\delta(T_p^{l_p})\over \delta(T_q^{l_q})}\in \Quot(R(A, P_0))_0.
    $$
    It follows from Lemma~\ref{QuotLem} that
    $$
    \Quot(R(A, P_0))_0 = \mathbb K\left({T_0^{l_0}\over T_1^{l_1}}\right).
    $$
    Thus, there exists a polynomial $F\in\mathbb K[u]$ with
    $$
        F\left({T_0^{l_0}\over T_1^{l_1}}\right) = \lambda \quad\text{in $\Quot(\mathbb L[T_{i1}]/\mathbb L \mathfrak I)$}.$$
    By the form of equations $\mathfrak g_i = 0$, one can check that the rational morphism $$\varphi\colon V\dashrightarrow \mathbb L,\quad (t_{01},\cdots,t_{r1})\mapsto {t_{01}^{l_{01}}\over t_{11}^{l_{11}}}{T_{02}^{l_{02}}\cdots T_{0n_0}^{l_{0n_0}}\over T_{12}^{l_{12}}\ldots T_{1n_1}^{l_{1n_1}}}$$ is surjective, so it defines the field embedding $$\varphi^*\colon \mathbb L(u)\hookrightarrow \Quot(\mathbb L[T_{i1}]/\mathbb L \mathfrak I), \quad u\mapsto {T_0^{l_0}\over T_1^{l_1}}.$$
    Consequently,
    $
    F(u) = \lambda
    $
    in $\mathbb L(u)$. However, $F\in\mathbb K[u]$ which implies $\lambda\in\mathbb K$.
    
    Hence,
   the elements $\delta(T_p^{l_p})$ and $\delta(T_q^{l_q})$ are proportional (with the coefficient $\lambda$) over the field~$\mathbb K$. As our argument is valid for any pair $p\neq q$, we have
    $$
    \dim_{\mathbb K} \Lin_{\mathbb K}\left\{ \delta(T_i^{l_i}) \,|\, i=0,\ldots,r\right\} =1.
    $$
\end{proof}

Now we are able to prove the main result for trinomial algebras of type~2.

\begin{theorem}\label{Type2Theorem}
    Let $R(A, P_0)$ be a trinomial algebra of type~2, and $\delta$ a $K_0$-homogeneous locally nilpotent derivation of $R(A, P_0)$. Then $\delta$ is elementary.
\end{theorem}

\begin{proof}
If $\delta$ is a non-proper derivation, it is elementary by Lemma~\ref{S-lemma}. Let $\delta$ be proper. According to Proposition~\ref{forDimStat}, 
    $$
    \dim \Lin\left\{ \delta(T_i^{l_i}) \,|\, i=0,\ldots,r\right\} =1,
    $$
    and Proposition~\ref{DimStat} asserts that $\delta$ is elementary.
\end{proof}

Theorems~\ref{Type1Theorem} and \ref{Type2Theorem} imply the following result.

\begin{theorem}\label{ExistenсeTheorem} Consider a trinomial algebra $R(A, P_0)$. 

    Let $R(A, P_0)$ be of type~$1$. Then $R(A, P_0)$ admits a nonzero $K_0$-homogeneous locally nilpotent derivation if and only if one of the following conditions is satisfied:
    \begin{enumerate}
        \item[$(1)$] $m>0$;
        \item[$(2)$] There is such $i_0\in\{ 1,\ldots,r \}$ that for all $i\in\{ 1,\ldots,r \}\setminus\{i_0\}$ there exist $c_i\in\{1,\ldots,n_i\}$ with $l_{ic_i}=1$.
    \end{enumerate}

    Let $R(A, P_0)$ be of type~$2$. Then $R(A, P_0)$ admits a nonzero $K_0$-homogeneous locally nilpotent derivation if and only if one of the following conditions is satisfied:
    \begin{enumerate}
        \item[$(1)$] $m>0$;
        \item[$(2)$] There are at most two such $i_0,\,i_1\in\{ 0,\ldots,r \}$ that for all $i\in\{0,\ldots,r \}\setminus\{i_0, i_1\}$ there exist $c_i\in\{1,\ldots,n_i\}$ with $l_{ic_i}=1$.
    \end{enumerate}
\end{theorem}

\begin{proof}
    According to Theorems~\ref{Type1Theorem} and \ref{Type2Theorem}, any nonzero $K_0$-homogeneous locally nilpotent derivation corresponds to a derivation described in Construction~\ref{ConstrDer}. The derivations ${\partial\over\partial S_p}$ exist if and only if $m>0$. The existence of derivations $\delta_C$ and $\delta_{C,\beta}$ is equivalent to the existence of tuples $C$ satisfying the conditions of Construction~\ref{ConstrDer}. Evidently, these conditions are exactly the conditions~(2) of the theorem.
\end{proof}

As an application of out results, we describe trinomial algebras that have non-zero locally nilpotent derivations but have no $K_0$-homogeneous derivations.

Denote the additive group of the ground field $\mathbb K$ by $\mathbb G_a$. In geometric terms, any locally nilpotent derivation of a finitely generated $\mathbb K$-domain $R$ corresponds to a $\mathbb G_a$-action on the affine variety $X=\Spec R$. An affine variety $X$ (and the algebra~$\mathbb K[X]$) is called \textit{rigid} if~$X$ admits no nontrivial $\mathbb G_a$-action, or, equivalently, $\mathbb K[X]$ admits no non-zero locally nilpotent derivation. In~\cite[Theorem~4]{EGSh}, a necessary and sufficient condition for a trinomial variety to be rigid is developed.

The $K_0$-grading defined in Construction~\ref{ConstrR} corresponds to the effective action of the quasitorus $H=\Hom_{\mathbb Z}(K_0, \mathbb K^\times)$ on $R(A, P_0)$ given by $hF = h(w)F$ for $h\in H, \, F\in R(A, P_0)_w,\, w\in K_0$. This action naturally induces an action on the trinomial variety $X = \Spec R(A, P_0)$, and $K_0$-homogeneous locally nilpotent derivations of $R(A, P_0)$ stay in one-to-one correspondence with $\mathbb G_a$-actions on $X$ that are normalized by $H$ in $\Aut(X)$.


Recall that a totally ordered abelian group is an abelian group $K$ endowed with a total order~$\le$ that is translation invariant:
$a\le b$ implies $a+c\le b+c$ for all $a, b, c\in K$; see~\cite[Chapter~1]{Freu}. In this case, any locally nilpotent derivation $\delta$ of a $K$-graded finitely generated $\mathbb K$-domain $R$ has a decomposition onto a sum of $K$-homogeneous derivations $$\delta = \sum_{i=1}^{s} \delta_i, \qquad \deg\delta_i = w_i,\qquad w_1<\ldots< w_s.$$ Moreover, the extreme summands $\delta_1$ and $\delta_s$ are also locally nilpotent. It follows that $R$ is rigid if and only if it admits no non-zero $K$-homogeneous derivations.

However, it is well-known that a finitely generated abelian group can be totally ordered if and only if it is torsion-free. We have seen in Example~\ref{ExampleVar} that in the case of the finest grading of a trinomial algebra, $K_0$ may not be a free abelian group. Thus, it is a natural question to ask when a trinomial algebra is not rigid and has no $K_0$-homogeneous locally nilpotent derivation, i.e., admits no $H$-normalized $\mathbb G_a$-actions. The answer is provided by the following result. It follows directly from Theorem~\ref{ExistenсeTheorem} compared to the rigidity criterion \cite[Theorem 4]{EGSh}.

        
\begin{cor}\label{Cor1}
    A trinomial algebra $R(A, P_0)$ is not rigid and admits no nontrivial $H$-normalized $\mathbb G_a$-action if and only if it is of type~2 and both the following conditions are fulfilled: 
    \begin{enumerate}
        \item[$(1)$] $m=0$;
        \item[$(2)$] There exist exactly three numbers $i_0, i_1, i_2\in\{0,\ldots,r\}$ such that for all $i\in\{0,\ldots,r\}\setminus\{i_0,i_1,i_2\}$ there are $c_i\in\{1\ldots,n_i\}$ with $\l_{ic_i}=1$. Furthermore, for all $i\in\{i_0,i_1\}$ there exist $c_i\in\{1,\ldots,n_i\}$ with $l_{ic_i}=2$; the exponents $l_{ij}$ are even for all $i\in\{i_0,i_1\},\,j\in\{1,\ldots,n_i\}$; and $l_{i_2 j}>1$ for all $j\in\{1,\ldots,n_{i_2}\}$.
    \end{enumerate}
\end{cor}

\section{Kernels of derivations from Construction~\ref{ConstrDer}}
In this section, we describe elements of the kernel of a derivation defined in Construction~\ref{ConstrDer}. Together with Theorems~\ref{Type1Theorem} and~\ref{Type2Theorem}, the following propositions give us an explicit description of all $K_0$-homogeneous locally nilpotent derivations of a trinomial algebra.

\begin{stat}\label{KerStat1}
    Suppose that $R(A, P_0)$ is of type~1 and $\delta$ a locally nilpotent derivation of $R(A, P_0)$ defined in Construction~\ref{ConstrDer}.
    
    Let $\delta = {\partial\over\partial S_p}$. Then $h$ is a $K_0$-homogeneous element of $\Ker\delta$ if and only if
    $$
    h = F(T_1^{l_1}) \Prod T_{ij}^{b_{ij}}\Prod_{k\neq p} S_k^{d_k}. 
    $$
    for some $b_{ij}, d_k\in\mathbb Z_{\ge 0}$ and some polynomial $F$ in one variable.

    Let $\delta = \delta_C$. Then $h$ is a $K_0$-homogeneous element of $\Ker\delta$ if and only if
    $$
    h = \alpha\Prod_{T_{ij}\in\Ker\delta} T_{ij}^{b_{ij}}\Prod S_k^{d_k} 
    $$
    for some $b_{ij}, d_k\in\mathbb Z_{\ge 0}$, $\alpha\in\mathbb K$.
\end{stat}

\begin{proof}
    We begin with the `only if' part. According to Lemma~\ref{KerLem1}, any $K_0$-homogeneous element $h\in R(A, P_0)$ has the form
    $$
    h = F(T_1^{l_1})\Prod T_{ij}^{b_{ij}}\Prod S_k^{d_k} 
    $$
    for some $b_{ij}, d_k\in\mathbb Z_{\ge 0}$ and a polynomial $F$ in one variable.

    The kernel of a locally nilpotent derivation is factorially closed (Lemma~\ref{BasicLem}). Consequently, $h\in\Ker\delta$ if and only if $F(T_1^{l_1})\in\Ker\delta$, $T_{ij}\in\Ker\delta$ for $b_{ij}>0$, and $S_k\in\Ker\delta$ for $d_k>0$.

    If $\delta = {\partial\over\partial S_p}$, then $F(T_1^{l_1})\in\Ker\delta$ for any polynomial $F$. In the other case, if $\delta = \delta_C$, we have $\delta(F(T_1^{l_1})) = F'(T_1^{l_1})\delta(T_1^{l_1}) \neq 0$ for $F\notin\mathbb K$. Here $F'$ denotes the formal derivative of $F$.
    
    The `if' part is clear: letting $h$ as given in the formulation of the proposition, we see that all the factors belong to $\ker\delta$, so $h\in\ker\delta$.
\end{proof}

\begin{stat}\label{KerStat2}
    Suppose that $R(A, P_0)$ is of type~2 and $\delta$ a locally nilpotent derivation of $R(A, P_0)$ defined in Construction~\ref{ConstrDer}.
    
    Let $\delta = {\partial\over\partial S_p}$.  Then $h$ is a $K_0$-homogeneous element of $\Ker\delta$ if and only if 
    $$
    h = F(T_1^{l_1}, T_2^{l_2}) \Prod T_{ij}^{b_{ij}}\Prod_{k\neq p} S_k^{d_k}. 
    $$
    for some $b_{ij}, d_k\in\mathbb Z_{\ge 0}$ and some homogeneous polynomial $F$ in two variables.
    
    Let $\delta = \delta_{C, \beta}$. Then $h$ is a $K_0$-homogeneous element of $\Ker\delta$ if and only if
    $$    
        h = \alpha(\beta_2T_1^{l_1} - \beta_1 T_2^{l_2})^{m} \Prod_{T_{ij}\in\Ker\delta} T_{ij}^{b_{ij}} \Prod S_k^{d_k}     
    $$
    for some $m,b_{ij},d_k\in\mathbb Z_{\ge 0},\alpha\in\mathbb K$.
\end{stat}

\begin{proof}
    According to Lemma~\ref{KerLem}, any $K_0$-homogeneous element of $h\in R(A, P_0)$ has the form
    $$
    h = F(T_1^{l_1}, T_2^{l_2})\Prod S_k^{d_k}\Prod T_{ij}^{b_{ij}}.
    $$ 
    
    In the case of $\delta = {\partial\over\partial S_p}$, the proof is the same as for the corresponding case of Proposotion~\ref{KerStat1}.
    
    Let $\delta = \delta_{C, \beta}$ and $h\in\Ker\delta$ be a $K_0$-homogeneous element. As the field $\mathbb K$ is algebraically closed, any homogeneous polynomial $F$ in two variables has the form $F(X_1,X_2) = \Prod\limits_{q=1}^t (\alpha_{q1}X_1 + \alpha_{q2} X_2)$ for some $t\in\mathbb Z_{\ge0}$, $\alpha_{q1},\alpha_{q2} \in\mathbb K$. Hence, 
$$
h = \Prod\limits_{q=1}^t (\alpha_{q1}T_1^{l_1} + \alpha_{q2} T_2^{l_2}) \Prod_k S_k^{d_k}\Prod_{(i,j)} T_{ij}^{b_{ij}}.
$$ 
The kernel of a locally nilpotent derivation is factorially closed by Lemma~\ref{BasicLem}. Thus we obtain  $S_k,T_{ij}\in\Ker\delta$ for $d_k> 0$ and $b_{ij}> 0$ respectively, and $\alpha_{q1}T_1^{l_1} + \alpha_{q2} T_2^{l_2}\in\Ker\delta$ for all~$q$. Recall that $\delta_{C, \beta}$ may have slightly different forms described in items~2.1 and~2.2 of Construction~\ref{ConstrDer}. We have
$$
0 = \delta(\alpha_{q1}T_1^{l_1} + \alpha_{q2} T_2^{l_2}) = 
	\begin{cases}
		(\alpha_{q1}\beta_1 + \alpha_{q2}\beta_2)\Prod\limits_k{\partial T_k^{l_k}\over\partial T_{kc_k}} &\text{for item~2.1},\\
		(\alpha_{q1}\beta_1 + \alpha_{q2}\beta_2)\Prod\limits_{k\neq i_0}{\partial T_k^{l_k}\over\partial T_{kc_k}} &\text{for item~2.2}.
	\end{cases}
$$
Therefore, $\alpha_{q1}\beta_1 + \alpha_{q2}\beta_2=0$. Hence, the vectors $(\alpha_{q1}, \alpha_{q2})$ and $(-\beta_2, \beta_1)$ are proportional and~$h$ has the required form.

The `if' part is the same as the one of Proposition~\ref{KerStat1}.
\end{proof}


\begin{cor}\label{Cor2}
    Let $\widehat\delta$ be a $K_0$-homogeneous locally nilpotent derivation of a trinomial algebra. Then $\widehat\delta = h\delta$ where $\delta$ is a derivation defined in Construction~\ref{ConstrDer}, and $h$ has the form given by Propositions~\ref{KerStat1} and~\ref{KerStat2}.
\end{cor}

\end{document}